\documentclass[11pt,reqno]{amsart}
\usepackage{amsmath,amssymb, amscd, amsmath, color}
\usepackage{mathrsfs, bm}
\usepackage{graphics}
\input epsf

\numberwithin{equation}{section}
\numberwithin{figure}{section}
\newtheorem{theorem}{Theorem}[section]
\newtheorem{lemma}[theorem]{Lemma}
\newtheorem{corollary}[theorem]{Corollary}

\theoremstyle{definition}
\newtheorem{definition}[theorem]{Definition}
\newtheorem{example}[theorem]{Example}
\theoremstyle{remark}
\newtheorem{remark}[theorem]{Remark}

\newtheorem{notation}[theorem]{Notation}

\def\Xint#1{\mathchoice
{\XXint\displaystyle\textstyle{#1}}
{\XXint\textstyle\scriptstyle{#1}}
{\XXint\scriptstyle\scriptscriptstyle{#1}}
{\XXint\scriptscriptstyle\scriptscriptstyle{#1}}
\!\int}
\def\XXint#1#2#3{{\setbox0=\hbox{$#1{#2#3}{\int}$ }
\vcenter{\hbox{$#2#3$ }}\kern-.6\wd0}}

\def\dashint{\Xint-}

\begin{document}

\title[H\"{o}lder Continuity of Weak Solutions of $p$-Laplacian PDEs with VMO Coefficients]{H\"{o}lder Continuity of Weak Solutions of $\bm{p}$-Laplacian PDEs with VMO Coefficients}

\author[C. S. Goodrich]{Christopher S. Goodrich}
\address{School of Mathematics and Statistics\\
UNSW Australia\\
Sydney, NSW 2052 Australia}
\email[Christopher S. Goodrich]{cgoodrich@creightonprep.org; c.goodrich@unsw.edu.au}
%\urladdr{http://www.math.unl.edu/~s-cgoodri4/}
\author[M. A. Ragusa]{M. Alessandra Ragusa}
\address{Dipartimento di Matematica e Informatica\\
Universit\'a di Catania\\
Catania, Italy\\
RUDN University, 6 Miklukho - Maklay St, Moscow, 117198 (Russia)
}
\email[Alessandra Ragusa]{maragusa@dmi.unict.it}
%\author[A. Scapellato]{Andrea Scapellato}
%\address{Dipartimento di Matematica e Informatica\\
%Universita di Catania\\
%Catania, Italy}
%\email[Andrea Scapellato]{scapellato@dmi.unict.it}
%\date{May 29, 2018}

\begin{abstract}
We consider solutions $u\in W^{1,p}\big(\Omega;\mathbb{R}^{N}\big)$ of the $p$-Laplacian PDE
\begin{equation}
\nabla\cdot\big(a(x)|Du|^{p-2}Du\big)=0,\notag
\end{equation}
for $x\in\Omega\subseteq\mathbb{R}^{n}$, where $\Omega$ is open and bounded.  More generally, we consider solutions of the elliptic system
\begin{equation}
\nabla\cdot\left(a(x)g'\big(a(x)|Du|\big)\frac{Du}{|Du|}\right)=0\text{, }x\in\Omega\notag
\end{equation}
as well as minimizers of the functional
\begin{equation}
\int_{\Omega}g\big(a(x)|Du|\big)\ dx.\notag
\end{equation}
In each case, the coefficient map $a\ : \ \Omega\rightarrow\mathbb{R}$ is only assumed to be of class $VMO(\Omega)\cap L^{\infty}(\Omega)$, which means that it may be discontinuous.  Without assuming that $x\mapsto a(x)$ has any weak differentiability, we show that $u\in\mathscr{C}_{\text{loc}}^{0,\alpha}(\Omega)$ for each $0<\alpha<1$.  The preceding results are, in fact, a corollary of a much more general result, which applies to the functional
\begin{equation}
\int_{\Omega}f\big(x,u,Du\big)\ dx\notag
\end{equation}
in case $f$ is only asymptotically convex.
\end{abstract}

\maketitle

\section{Introduction}

In this note we consider a number of interrelated problems in regularity theory.  More specifically, we consider the $p$-Laplacian PDE
\begin{equation}\label{eq1.1}
\nabla\cdot\big(a(x)|Du|^{p-2}Du\big)=0
\end{equation}
on the open and bounded set $\Omega\subseteq\mathbb{R}^{n}$.  Our primary assumption on the coefficient function $x\mapsto a(x)$ is that
\begin{equation}\label{eq1.1mmm}
a\in VMO(\Omega)\cap L^{\infty}(\Omega).
\end{equation}
Since $u\ : \ \Omega\rightarrow\mathbb{R}^{N}$ and, hence, the codomain of $Du$ is $\mathbb{R}^{N\times n}$, we consider here the \emph{systems} case.  More generally, our results concern the elliptic PDE
\begin{equation}\label{eq1.2}
\nabla\cdot\left(a(x)g'\big(a(x)|Du|\big)\frac{Du}{|Du|}\right)=0\text{, }x\in\Omega,
\end{equation}
where $g\in\mathscr{C}^2(\mathbb{R})$.  Since our approach is variational in nature, we actually obtain results to problems \eqref{eq1.1}--\eqref{eq1.2} by means of studying the variational problem
\begin{equation}\label{eq1.3}
\int_{\Omega}g\big(a(x)|Du|\big)\ dx.
\end{equation}
Again, since the codomain of $u$ is $\mathbb{R}^{N}$, it follows that \eqref{eq1.3} is studied in the \emph{vectorial} setting.  In fact, the results regarding problems \eqref{eq1.1}--\eqref{eq1.3} are corollaries to a more general result, which is to study the regularity of minimizers of the integral functional
\begin{equation}\label{eq1.3a}
\int_{\Omega}f(x,u,Du)\ dx
\end{equation}
in the case where, other than a standard $p$-growth condition, we assume only that $f$ is asymptotically related to the map $(x,\xi)\mapsto g\big(a(x)|\xi|\big)$ in the sense that for each $\varepsilon>0$ there is a function $\sigma_{\varepsilon}\ : \ \Omega\rightarrow\mathbb{R}$ such that
\begin{equation}\label{eq1.4}
\Big|f(x,u,\xi)-g\big(a(x)|\xi|\big)\Big|<\varepsilon|\xi|^p
\end{equation}
whenever $|\xi|>\sigma_{\varepsilon}(x)$.  Roughly speaking, then condition \eqref{eq1.4} states that
\begin{equation}\label{eq1.5}
\int_{\Omega}f(x,u,Du)\ dx\approx\int_{\Omega}g\big(a(x)|Du|\big)\ dx\notag
\end{equation}
whenever $|Du|\gg1$.  Then minimizers of problem \eqref{eq1.3a} can be suitably compared to minimizers of problem \eqref{eq1.3}.

Problems similar to \eqref{eq1.1}--\eqref{eq1.3} were considered very recently in a paper by Goodrich \cite{goodrich5}.  There the author considered these problems under the assumption that the coefficient function $t\mapsto a(t)$ satisfied
\begin{equation}\label{eq1.1nnn}
a\in W_{\text{loc}}^{1,q}(\Omega)\cap L^{\infty}(\Omega),
\end{equation}
where $q>1$ was essentially arbitrary, and demonstrated the partial H\"{o}lder continuity of $u$ (as well as that of the gradient $Du$).  Furthermore, the singular set was demonstrated to be intimately connected with the points $x_0\in\Omega$ at which
\begin{equation}
\liminf_{R\to0^+}R^{q-n}\int_{\mathcal{B}_R\left(x_0\right)}|Da|^q\ dx>0.\notag
\end{equation}
More precisely, it was shown that under assumption \eqref{eq1.1nnn} either a solution of \eqref{eq1.2} or a minimizer of \eqref{eq1.3} satisfies
\begin{equation}
u\in\mathscr{C}_{\text{loc}}^{0,\alpha}\big(\Omega_0\big),\notag
\end{equation}
where singular set $\Omega\setminus\Omega_0$ satisfies
\begin{equation}
\Omega\setminus\Omega_0\subseteq\left\{x\in\Omega\ : \ \liminf_{R\to0^+}R^q\dashint_{\mathcal{B}_R}|Da|^q\ dx>0\right\}.\notag
\end{equation}
Thus, even though $\big|\Omega\setminus\Omega_0\big|=0$ we see that $\Omega\setminus\Omega_0$ may be nonempty.  This is interesting, in part, because since functional \eqref{eq1.3} has no $u$ dependence in the integrand (rather, only $x$ and $Du$ dependence) it would not be unexpected if the singular set was not only Lebesgue null but also empty.  Yet this was not able to be deduced in \cite{goodrich5}.

A natural question, then, is whether this same sort of curiosity is exhibited if we require, instead, that the coefficient $t\mapsto a(t)$ belongs to the space of VMO functions -- or, alternatively, if when $a\in VMO(\Omega)$, it then follows that the singular set is, in fact, empty.  We note that despite the generality of the result in \cite{goodrich5} the possibility of VMO-type coefficients was not considered.  Since it is known that
\begin{equation}
VMO(\Omega)\supsetneq W^{1,n}(\Omega),\notag
\end{equation}
it follows that there are VMO coefficients for which the results of \cite{goodrich5} do not apply.  As we demonstrate in this paper it turns out that by using the VMO assumption \eqref{eq1.1mmm} instead of the Sobolev assumption \eqref{eq1.1nnn} we obtain an empty singular set, which stands in considerable contrast to the results of \cite{goodrich5}.

So, with this context in mind, our main contribution is to extend the results of \cite{goodrich5} to study the regularity of $u$ in problems \eqref{eq1.1}--\eqref{eq1.3} for $a\in VMO(\Omega)$, thereby showing that the case of VMO-type coefficients results in a sharper result than the more general Sobolev assumption.  To do this we first treat the more general asymptotically convex problem \eqref{eq1.3a}.  In considering these problems, we do not necessarily assume that the coefficient map satisfies any requirement of differentiability -- either classical or weak.  Moreover, since, for $0<\theta<1$, it holds that $W^{\theta,n/\theta}(\Omega)\subseteq VMO(\Omega)$ (see Remark \ref{remark2.11}), it follows that due to (1) our results can apply to cases in which the coefficient $x\mapsto a(x)$ belongs to a suitable fractional Sobolev space.

To conclude this section we would like to provide a brief survey of the relevant literature.  Regarding VMO-type coefficients, many authors have considered such problems in the past.  Among these are papers by B\"{o}gelein, et al. \cite{bogeleinduzaar1}, Di Fazio et al. \cite{difazio1,difazio2,difazio3}, Dan\u{e}\u{c}ek and Viszus \cite{danecek1}, Goodrich \cite{goodrich1}, Dong and Gallarati \cite{dong1}, Ragusa \cite{ragusa3,ragusa4,ragusa5}, Ragusa and Tachikawa \cite{ragusa1,ragusa2}, Tan, Wang, and Chen \cite{tan1}, Wang and Manfredi \cite{wang1}, and Wang, Liao, Zhu, Liao, and Hong \cite{wang2}.  However, other than \cite{goodrich1} above, we are not aware of any that considers the asymptotically convex setting with any sort of VMO-type coefficients.  And in the case of \cite{goodrich1} the result deduced there was not necessarily associated to an empty singular set -- cf., \cite[Theorem 3.1]{goodrich1}.  Consequently, it is not clear from the results of that paper what should happen in the cases considered in this paper.  Also in the context of weak requirements on the coefficients, we would like to mention the paper by Foss and Mingione \cite{fossmingione1}, which presented the first results in the literature of partial H\"{o}lder continuity of solutions to elliptic systems and quasiconvex functionals under the assumption of only continuity of the coefficients.

More generally, the asymptotically convex setting (in the non-VMO coefficient case) has been considered in many contexts.  First by Chipot and Evans \cite{chipot1}, then Raymond \cite{raymond1}, and later Diening, Stroffolini, and Verde \cite{diening2}, Fey and Foss \cite{feyfoss1,feyfoss2}, Foss \cite{foss1}, Foss, Passarelli di Napoli, and Verde \cite{fossnapoliverde1,fossnapoliverde2}, Foss and Goodrich \cite{fossgoodrich1,fossgoodrich2}, Goodrich \cite{goodrich3,goodrich2,goodrich4}, Goodrich and Scapellato \cite{goodrichscapellato1}, and Scheven and Schmidt \cite{scheven1,scheven2}.  Furthermore, there are also many recent papers that consider discontinuous coefficients of some sort in problems such as \eqref{eq1.2} or \eqref{eq1.3}.  Besides the already mentioned recent paper by Goodrich \cite{goodrich5}, which is our primary motivation here, many other related papers have appeared in recent years such as those by B\"{o}gelein, Duzaar, Habermann, and Scheven \cite{bogeleinduzaar2}, Cupini, Giannetti, Giova, and Passarelli di Napoli \cite{cupini1}, Giova and Passarelli di Napoli \cite{giova1}.

We would also like to point out that in addition to the aforementioned papers the classical papers by De Giorgi \cite{degiorgi1}, Giaquinta and Giusti \cite{giaquintagiusti1,giaquintagiusti2}, Giaquinta and Modica \cite{giaquintamodica1,giaquintamodica2,giaquintamodica3,giaquintamodica4}, Schoen and Uhlenbeck \cite{uhlenbeck2}, and Uhlenbeck \cite{uhlenbeck1} are useful references for background on regularity theory and its development.  Moreover, two papers by Mingione \cite{mingione1,mingione2} are useful in this regard as well.

So, in this paper we extend the preceding results by, in particular, showing that the recent of results of Goodrich \cite{goodrich5} can be extended to the VMO case.  And, in particular, we obtain a cleaner and sharper result in this special case.

\section{Preliminaries}

In this section we collect some preliminary results for later use.  We note that the monographs by Giaquinta \cite{giaquinta1} and Giusti \cite{giusti1} are references for additional background related to the results we present here.  We begin by mentioning the notation we employ in this work.

\begin{notation}\label{notation2.1}
We will adhere to the following notational conventions.
\begin{itemize}
\item The number $C$ will be a generic constant, whose value may vary from line to line without specific mention, and we will always assume, without loss of generality, that $C\ge1$.  While we will not generally denote any other quantities on which $C$ depends, we remark that $C$ will never depend on the radius of any ball, $R$, used as an integration set.

\item Given a point $x_0$ of $\mathbb{R}^{n}$ and a real number $R>0$, by $\mathcal{B}_{R}\left(x_0\right)$ we denote the open set
\begin{equation}
\mathcal{B}_{R}\left(x_0\right):=\big\{x\in\mathbb{R}^{n}\ : \ \left|x-x_0\right|<R\big\}.\notag
\end{equation}
Generally the center of the ball will be clear from context typically.  Consequently, we will typically write $\mathcal{B}_R$ for $\mathcal{B}_R\left(x_0\right)$.

\item To denote the average value of a function $f\ : \ \Omega\subset\mathbb{R}^n\rightarrow\mathbb{R}$ we write
\begin{equation}
\dashint_{E}f(x)\ dx:=\frac{1}{|E|}\int_{E}f(x)\ dx,\notag
\end{equation}
where $E\subset\Omega$ is assumed to be Lebesgue measurable and $|E|$ denotes the Lebesgue measure of the set $E$.  In addition, given a point $x_0\in\Omega\subseteq\mathbb{R}^{n}$, a real number $R>0$, and a function $u\ : \ \Omega\rightarrow\mathbb{R}^{N}$, by the symbol $(u)_{x_0,\rho}\in\mathbb{R}^{N}$ we denote the vector (or real number if $N=1$)
\begin{equation}
(u)_{x_0,R}:=\dashint_{\mathcal{B}_{R}\left(x_0\right)}u(x)\ dx.\notag
\end{equation}
Since the center, $x_0$, of the ball generally will be clear from the context, we will typically write $(u)_R$ to denote $(u)_{x_0,R}$.
\end{itemize}
\end{notation}

Next we state the various function spaces that we use in our arguments.

\begin{definition}\label{definition2.2}
For $p\in[1,+\infty)$ and $0\le\gamma\le n$, the \textbf{\emph{Morrey space}}, denoted by $L^{p,\gamma}\left(E;\mathbb{R}^{N}\right)$, is defined by
\begin{equation}\label{eq2.1}
L^{p,\gamma}\left(E;\mathbb{R}^N\right):=\left\{u\in L^p\left(E;\mathbb{R}^{N}\right)\ : \ \sup_{\substack{y\in E\\ \rho>0}}\frac{1}{\rho^{\gamma}}\int_{E\cap\mathcal{B}_{\rho}(y)}|u|^p\ dx<+\infty\right\},\notag
\end{equation}
where $E\subseteq\mathbb{R}^n$ is a measurable set.  Note that we will write $u\in L_{\text{loc}}^{p,\gamma}\left(E;\mathbb{R}^N\right)$ if $u\in L^{p,\gamma}\left(E';\mathbb{R}^N\right)$ for each $E'\Subset E$.
\end{definition}

\begin{definition}\label{definition2.2a}
For $p\in[1,+\infty)$ and $0\le\gamma\le n$, the \textbf{\emph{Campanato space}}, denoted by $\mathscr{L}^{p,\gamma}\left(E;\mathbb{R}^{N}\right)$, is defined by
\begin{equation}\label{eq2.1a}
\mathscr{L}^{p,\gamma}\left(E;\mathbb{R}^N\right):=\left\{u\in L^p\left(E;\mathbb{R}^{N}\right)\ : \ \sup_{\substack{y\in E\\ \rho>0}}\frac{1}{\rho^{\gamma}}\int_{E\cap\mathcal{B}_{\rho}(y)}\left|u-(u)_{y,\rho}\right|^p\ dx<+\infty\right\},\notag
\end{equation}
where $E\subseteq\mathbb{R}^n$ is a measurable set.  Note that we will write $u\in \mathscr{L}_{\text{loc}}^{p,\gamma}\left(E;\mathbb{R}^N\right)$ if $u\in\mathscr{L}^{p,\gamma}\left(E';\mathbb{R}^N\right)$ for each $E'\Subset E$.
\end{definition}

\begin{definition}\label{definition2.3}
For $p\in[1,+\infty)$ and $0\le\kappa\le n$, a function $u\in W^{1,p}\left(E,\mathbb{R}^N\right)$, where $E\subseteq\mathbb{R}^{n}$, is said to belong to the \textbf{\emph{Sobolev-Morrey space}} $W^{1,(p,\kappa)}\left(E;\mathbb{R}^N\right)$ provided that $u\in L^{p,\kappa}\left(E;\mathbb{R}^N\right)$ and $Du\in L^{p,\kappa}\left(E;\mathbb{R}^N\right)$.  Note that we will write $u\in W_{\text{loc}}^{1,(p,\kappa)}\left(E;\mathbb{R}^N\right)$ provided that $u\in W^{1,(p,\kappa)}\left(E';\mathbb{R}^N\right)$ for each $E'\Subset E$.
\end{definition}

In 1961 John and Nirenberg in \cite{K2} defined the space of functions having bounded mean oscillation.

\begin{definition}\label{BMO}
We say that $f\in L_{\mathrm{loc}}^1({\mathbb R}^n)$ belongs to the space $BMO$ (\textit{\textbf{bounded mean oscillation}}) if
$$\|f\|_* \equiv \sup_{\mathcal{B}} \frac{1}{|\mathcal{B}|}\int_{\mathcal{B}}|f(x)-f_{\mathcal{B}}|\ dx<+\infty$$
where $\mathcal{B}$ ranges in the class of the balls of ${\mathbb R}^n$ and
$$f_\mathcal{B}\equiv \frac{1}{|\mathcal{B}|}\int_{\mathcal{B}} f(x)\ dx$$
is the integral average of $f$ in $\mathcal{B}$.
\end{definition}

\begin{remark}\label{remark2.6}
We would like to point out that $\|f\|_*$ is a norm in $BMO$ modulo constant functions under which $BMO$ is a Banach space (see \cite{K4}).
\end{remark}

\begin{remark}\label{remark2.6aaa}
In the proofs of our results (see, especially, the proof of Theorem \ref{theorem3.1}) we will estimate from above quantities of the form
\begin{equation}
\int_{\mathcal{B}_R}\big|a(x)-(a)_R\big|\ dx\notag
\end{equation}
by the BMO norm $\Vert a\Vert_*$.  When we do this, it is understood that the supremum in Definition \ref{BMO} is taken over balls $\mathcal{B}\subseteq\mathcal{B}_R$.  We will do this without any particular mention of this fact within the proofs of Section 3.
\end{remark}

Next we recall an important lemma that we extensively use in the sequel.

\begin{lemma}[\cite{K2}]\label{lemma2.6}
Let $b\in BMO$. Then, for any $1\leq p<\infty$, there exists a positive constant $C$ such that
$$\sup_{\mathcal{B}} \left(\frac{1}{|\mathcal{B}|}\int_{\mathcal{B}}|b(x)-b_{\mathcal{B}}|^p\ dx\right)^{\frac{1}{p}}\leq C\|b\|_*.$$
\end{lemma}

In 1975 Sarason in \cite{K7} defined a new class of functions.  This is in Definition \ref{vmo}.

\begin{definition}[\cite{K7}]\label{vmo}
Let $f\in BMO$ and $r>0$. We set
\begin{equation}
\eta(r)\equiv \sup_{\rho\leq r}\frac{1}{|\mathcal{B}|}\int_{\mathcal{B}} |f(x)-f_{\mathcal{B}}|\ dx,\notag
\end{equation}
where $\mathcal{B}$ ranges in the class of the balls of ${\mathbb R}^n$ having radius $\rho$ less than or equal to $r$.  We say that  $f\in BMO$ belongs to the class $VMO$ (\textit{\textbf{vanishing mean oscillation}}) if
$$\lim_{r\to 0^+}\eta(r)=0.$$
In addition, we say that the function $\eta(r)$ is the \textbf{\textit{VMO modulus}} of $f$.
\end{definition}

We note that all functions in the Sobolev spaces $W^{1,n}$ belong to VMO; this fact follows directly from the Poincar\'e inequality. Indeed,
\begin{equation}
\frac{1}{|\mathcal{B}|}\int_{\mathcal{B}}\big|f(x)-f_{\mathcal{B}}\big|\ dx\leq c(n)\left(\int_{\mathcal{B}} |\nabla f(x)|^n\ dx\right)^{1/n}.\notag
\end{equation}

The following Sarason's characterization of $VMO$ holds.

\begin{theorem}[\cite{K7}]
Let $f\in BMO$. The following conditions are equivalent:
\begin{enumerate}
	\item $f$ belongs to $VMO$;
	
	\item $f$ belongs to the $BMO$ closure of the space of the bounded uniformly continuous functions;
	
	\item \label{third} $\displaystyle \lim_{h\to 0}\big\|f(\cdot-h)-f(\cdot)\big\|_*=0$.
\end{enumerate}
\end{theorem}

\begin{remark}\label{remark2.10}
Let us observe that (\ref{third}) implies the good behavior of $VMO$ functions. Precisely, if $f\in VMO$ and $\eta_f(r)$ is its $VMO$ modulus, we can find a  sequence $\{f_h\}_{h\in\mathbb{N}}\in C^\infty({\mathbb{R}}^n)$ of functions with $VMO$ moduli $\eta_{f_h}(r)$, such that $f_h\to f$ in $BMO$ as $h\to +\infty$ and $\eta_{f_h}(r)\leq \eta_f(r)$ for all integers $h$.
\end{remark}

\begin{remark}\label{remark2.11}
It easy to see that $W^{\theta,n/\theta}({\mathbb R}^n)$, $0<\theta<1$, is contained in $VMO$. Indeed, for any ball $\mathcal{B}\subset\mathbb{R}^{n}$ we calculate
\begin{eqnarray*}
\frac{1}{|\mathcal{B}|}\int_{\mathcal{B}}\big|f(x)-f_{\mathcal{B}}\big|\ dx &=& \left(\frac{1}{|\mathcal{B}|}\int_{\mathcal{B}}\big|f(x)-f_{\mathcal{B}}\big|^{n/\theta}\ dx\right)^{\theta/n}\\
&=& \left(\frac{1}{|\mathcal{B}|}\int_{\mathcal{B}}\left|\frac{1}{|\mathcal{B}|}\int_{\mathcal{B}}\big(f(x)-f(y)\big)\ dy\right|^{n/\theta}\ dx\right)^{\theta/n}\\
&\leq& \left(\frac{1}{|\mathcal{B}|}\int_{\mathcal{B}}\frac{1}{|\mathcal{B}|}\int_{\mathcal{B}}\big|f(x)-f(y)\big|^{n/\theta}\ dx\ dy\right)^{\theta/n}\\
&\leq& c_{n,\theta}\left(\int_{\mathcal{B}}\int_{\mathcal{B}}\frac{|f(x)-f(y)|^{n/\theta}}{|x-y|^{2n}}\ dx\ dy\right)^{\theta/n},
\end{eqnarray*}
where the constant $c_{n,\theta}$ depends both on $n$ and $\theta$.  Because $f\in W^{\theta,n/\theta}({\mathbb R}^n)$ it follows that
$$\left(\int_{{\mathbb R}^n}\int_{{\mathbb R}^n}\frac{\big|f(x)-f(y)\big|^{n/\theta}}{|x-y|^{2n}}\ dx\ v  dy\right)^{\theta/n}<+\infty,$$
which, thanks to the absolute continuity of the integral, we conclude that
\begin{equation}
W^{\theta,n/\theta}\big(\mathbb{R}^{n}\big)\subseteq VMO\big(\mathbb{R}^{n}\big),\notag
\end{equation}
as desired.
\end{remark}

We next state precisely the hypotheses imposed on the problems we study in this work.  These hypotheses essentially were already mentioned in Section 1.

\begin{list}{}{\setlength{\leftmargin}{.5in}\setlength{\rightmargin}{0in}}
\item[\textbf{A1:}] The function $g\ : \ \mathbb{R}\rightarrow\mathbb{R}$ satisfies each of the following conditions for real constants $0<c_1<c_2$ and $0<c_3<c_4$ and $0<c_5<c_6$ and $p\ge2$.
\begin{enumerate}
\item $c_1t^p\le g(t)\le c_2\left(1+t^p\right)$

\item $c_3t^{p-1}\le g'(t)\le c_4\left(1+t^{p-1}\right)$

\item $c_5t^{p-2}\le g''(t)\le c_6\left(1+t^{p-2}\right)$

\item $g\in\mathscr{C}^2(\mathbb{R})$
\end{enumerate}

\item[\textbf{A2:}] The function $a\ : \ \Omega\rightarrow(0,+\infty)$ satisfies each of the following conditions for real constants $0<a_1<a_2<+\infty$.
\begin{enumerate}
\item $a\in VMO(\Omega)\cap L^{\infty}(\Omega)$

\item $a_1\le a(x)\le a_2$, for a.e. $x\in\Omega$
\end{enumerate}

\item[\textbf{A3:}] The function $f\ : \ \Omega\times\mathbb{R}^{N}\times\mathbb{R}^{N\times n}\rightarrow\mathbb{R}$ satisfies the following conditions.
\begin{enumerate}
\item For each $\varepsilon>0$ there exists a function $\sigma_{\varepsilon}\in L^{\gamma,\beta}(\Omega)$, where $n-p<\beta<n$ and $\gamma$ is a number satisfying $\gamma>p$, such that
\begin{equation}
\Big|f(x,u,\xi)-g\big(a(x)|\xi|\big)\Big|<\varepsilon|\xi|^p,\notag
\end{equation}
whenever $|\xi|>\sigma_{\varepsilon}(x)$, where
\begin{enumerate}
\item the function $t\mapsto g(t)$ satisfies the hypotheses of condition (A1) above; and

\item the function $x\mapsto a(x)$ satisfies the hypotheses of condition (A2) above.
\end{enumerate}

\item $\big|f(x,u,\xi)\big|\le \mu(x)+C\big(1+|\xi|^2\big)^{\frac{p}{2}}$, for each $(x,u,\xi)\in\Omega\times\mathbb{R}^{N}\times\mathbb{R}^{N\times n}$, and where $\mu\ : \ \Omega\rightarrow\mathbb{R}$ is a function satisfying
\begin{equation}
\mu\in L^{\eta_0,\beta}(\Omega),\notag
\end{equation}
for some $\eta_0>1$.
\end{enumerate}
\end{list}

\begin{remark}\label{remark2.12}
Due to Remark \ref{remark2.11}, coupled with condition (A2.1), we see that our results here apply to problems \eqref{eq1.1}--\eqref{eq1.3} in the case where $a\in W^{\theta,n/\theta}(\Omega)$, for $0<\theta<1$, since
\begin{equation}
a\in VMO(\Omega)\supseteq W^{\theta,n/\theta}(\Omega).\notag
\end{equation}
\end{remark}

Next we mention a classical reverse H\"{o}lder inequality for a local minimizer of the functional \eqref{eq1.3}.

\begin{lemma}\label{lemma2.1}
Let $u\in W_{\text{loc}}^{1,p}(\Omega)$ be a local minimizer of \eqref{eq1.3}.  Then there exist a number $\varepsilon_0>0$ and a positive constant $C:=C\big(n,p,\varepsilon_0\big)$ such that for each number $0<\varepsilon<\varepsilon_0$ it holds that
\begin{equation}
\left(\dashint_{\mathcal{B}_{\frac{R}{2}}}|Du|^{p+\varepsilon}\ dx\right)^{\frac{p}{p+\varepsilon}}\le\dashint_{\mathcal{B}_R}|Du|^p\ dx,\notag
\end{equation}
for every $\mathcal{B}_{\frac{R}{2}}\subset\mathcal{B}_R\subset\Omega$.
\end{lemma}

We next state a coercivity-type result for the functional \eqref{eq1.3a}.  This result is useful in proving an appropriate reverse H\"{o}lder-type inequality for the asymptotic problem.  Note that the number $\beta$ in the statement of Lemma \ref{lemma2.16} is the same the number $\beta$ in condition (A3).

\begin{lemma}\label{lemma2.16}
Assume that conditions (A2)--(A3) hold.  Then there is a constant $C$, depending only upon initial data, such that $f$ satisfies the coercivity condition
\begin{equation}
f(x,u,\xi)\ge\frac{1}{2}C|\xi|^p-A(x),\notag
\end{equation}
where the function $A\ : \ \Omega\rightarrow[0,+\infty)$ is defined by
\begin{equation}
A(x):=C\Big[\mu(x)+\big(\sigma_{\varepsilon}(x)\big)^p+1\Big]\notag
\end{equation}
and satisfies
\begin{equation}
A\in L^{\eta_0,\beta}(\Omega),\notag
\end{equation}
for some number $\eta_0>1$; here $x\mapsto\mu(x)$ is the same function as in condition (A3.2).
\end{lemma}

\begin{proof}
The proof of this theorem is essentially the same as \cite[Lemma 2.11]{goodrich4} -- see also \cite{goodrich1}.  Therefore, we omit it.
\end{proof}

With the preceding coercivity lemma in hand, we next state a reverse H\"{o}lder inequality for any function $u$, which minimizes functional \eqref{eq1.3a}.  This inequality is technically distinct from Lemma \ref{lemma2.1} since that lemma concerned a minimizer for the functional \eqref{eq1.3}.  Since a minimizer of \eqref{eq1.3a} is in the asymptotic setting, the associated reverse H\"{o}lder inequality is slightly altered.

\begin{lemma}\label{lemma2.16aaa}
Suppose that $u\in W^{1,p}(\Omega)$ is a minimizer of the functional \eqref{eq1.3a}.  Then there exists a number $r_0>1$ such that for each number $r\in\left(1,r_0\right)$ the inequality
\begin{equation}\label{eq2.8}
\dashint_{\mathcal{B}_R}|Du|^q\ dx\le C\left(\dashint_{\mathcal{B}_{2R}}|Du|^p\ dx\right)^{\frac{q}{p}}+C\dashint_{\mathcal{B}_{2R}}\big(A(x)+1\big)^{\frac{q}{p}}\ dx\notag
\end{equation}
is satisfied, where $q:=pr$, $\mathcal{B}_{2R}\Subset\Omega$, and the function $x\mapsto A(x)$ is as in Lemma \ref{lemma2.16}.
\end{lemma}

\begin{proof}
Since this is essentially the same as the analogous lemmata in, for example, \cite{fossgoodrich1,goodrich1,goodrich4}, we omit the proof.
\end{proof}

\begin{remark}
We note that due to the presence of the function $x\mapsto A(x)$ in the higher integrability result of Lemma \ref{lemma2.16aaa}, it follows that the number $r_0$, which is the upper bound on the amount of higher integrability that $u$ enjoys, depends on the structural parameters of the problem -- e.g., the number $\eta_0$ appearing in condition (A3.2).
\end{remark}

We next state a suitable reverse H\"{o}lder inequality in the boundary case.  In particular, as part of our proofs in Section 3 we will compare a minimizer $u$ of \eqref{eq1.3a}, namely
\begin{equation}
\int_{\Omega}f(x,u,Du)\ dx,\notag
\end{equation}
to a minimizer $v$ of the \textquotedblleft frozen functional\textquotedblright\
\begin{equation}
\int_{\mathcal{B}_R}g\big((a)_R|Dv|\big)\ dx.\notag
\end{equation}
The function $v$ will satisfy the boundary data $v\equiv u$ on $\partial\mathcal{B}_R$.  Consequently, we will need reverse H\"{o}lder inequality for $v$ that preserves the boundary data.  This is the content of the next lemma, and it is similar to analogous reverse H\"{o}lder inequalities in \cite{fossgoodrich1,goodrich1,goodrich4}.  Furthermore, it is similar to a boundary reverse H\"{o}lder inequality mentioned in \cite{goodrich5}.

\begin{lemma}\label{lemma2.17}
Let $w\in W^{1,p}\big(\mathcal{B}_R\left(x_0\right)\big)$ for a ball $\mathcal{B}_R$, where $\mathcal{B}_{2R}\left(x_0\right)\Subset\Omega$, such that $w$ is a minimizer for the frozen functional
\begin{equation}
w\mapsto\int_{\mathcal{B}_R\left(x_0\right)}g\big((a)_R|Dw|\big)\ dx,\notag
\end{equation}
where $x\mapsto a(x)$ satisfies condition (A2).  In addition, assume that $u-w\in W_0^{1,p}\left(\mathcal{B}_R\right)$, where $u\in W^{1,t}\left(\mathcal{B}_{2R}\right)$ for some number $t>p$.  Then there exists a number $r_0>1$ such that for each $r\in\big(1,r_0\big)$
\begin{equation}
\dashint_{\mathcal{B}_R}|Du-Dw|^q\ dx\le C\left(\dashint_{\mathcal{B}_R}|Du-Dw|^p\ dx\right)^{\frac{q}{p}}+C\dashint_{\mathcal{B}_{2R}}\big(1+|Du|^q\big)\ dx,\notag
\end{equation}
where $q:=pr$.
\end{lemma}

\begin{proof}
Since the proof is similar to that \cite[Lemma 2.15]{goodrich4}, we will only outline the essential differences.  As in that proof, we let
\begin{equation}
\frac{R}{2}<t<s\le R,\notag
\end{equation}
and we let $\eta\in\mathscr{C}^{\infty}\big(\mathcal{B}_s\big)$ be a cut-off function such that, in the usual way, we require that $\displaystyle|D\eta|\le\frac{2}{s-t}$ and
\begin{equation}
\eta(x)\equiv\begin{cases} 0\text{, }&x\in\mathcal{B}_{R}\setminus\mathcal{B}_{s}\\ 1\text{, }&x\in\mathcal{B}_{t}\end{cases}.\notag
\end{equation}
As in the statement of the lemma, let $w$ be a minimizer of the frozen functional
\begin{equation}
\int_{\mathcal{B}_R\left(x_0\right)}g\big((a)_R|Dw|\big)\ dx,\notag
\end{equation}
satisfying the boundary data
\begin{equation}
w\equiv u\text{, }x\in\partial\mathcal{B}_R.\notag
\end{equation}

Defining $\varphi:=\eta(w-u)$ and $v:=w-\varphi$, we can now basically follow the proof of \cite[Lemma 2.15]{goodrich4}, modulo one minor alteration.  In order to estimate the quantity $\displaystyle\int_{\mathcal{B}_s}|Dw|^p\ dx$ we write
\begin{equation}
\begin{split}
C\int_{\mathcal{B}_s}|Dw|^p\ dx&\le C\int_{\mathcal{B}_s}(a)_R|Dw|^p\ dx\\
&\le C\int_{\mathcal{B}_s}g\big((a)_R|Dw|\big)\ dx\\
&\le C\int_{\mathcal{B}_s}g\big((a)_R|Dv|\big)\ dx\\
&\le C\int_{\mathcal{B}_s}\Big(1+(a)_R^p|Dv|^p\Big)\ dx\\
&\le C\int_{\mathcal{B}_s}\big(1+|Dv|^p\big)\ dx\\
&\le C\int_{\mathcal{B}_s}1\ dx+C\int_{\mathcal{B}_s}|Du|^p\ dx+\frac{C}{(s-t)^p}\int_{\mathcal{B}_s}|u-w|^p\ dx\\
&\quad\quad\quad\quad\quad\quad\quad\quad\quad\quad\quad\quad\quad\quad\quad\quad+C\int_{\mathcal{B}_s}|1-\eta|^p|Du-Dw|^p\ dx,\notag
\end{split}
\end{equation}
where we have used the minimality of $w$.  Note that $v$ is, in fact, a valid competitor relative to $w$ since $v\equiv w$ on $\mathcal{B}_s$ with $v\in W^{1,p}\left(\mathcal{B}_R\right)$.  Then using this estimate together with the calculations in \cite[Lemma 2.15]{goodrich4} we obtain
\begin{equation}
\begin{split}
\int_{\mathcal{B}_t}|D\varphi|^p\ dx&\le C\int_{\mathcal{B}_s}1\ dx+C\int_{\mathcal{B}_s}|Du|^p\ dx+\frac{C}{(s-t)^p}\int_{\mathcal{B}_s}|u-w|^p\ dx\\
&\quad\quad\quad\quad\quad\quad\quad\quad\quad\quad\quad\quad\quad\quad\quad\quad+C\int_{\mathcal{B}_{s}\setminus\mathcal{B}_{t}}|Du-Dw|^p\ dx.\notag
\end{split}
\end{equation}
Finally, an application of Widman's \textquotedblleft hole filling trick\textquotedblright\ \cite{widman1} together with a result of Stredulinsky \cite[p. 409]{stredulinsky1} completes the proof and yields the desired inequality.
\end{proof}

Finally, the following results will prove useful in our regularity arguments.  These lemmata concern relationships between the Sobolev-Morrey spaces and the H\"{o}lder spaces, as well as Campanato spaces and H\"{o}lder spaces.  We note that Lemma \ref{lemma2.10} may be obtained from \cite[Theorem 2.9]{giusti1}.

\begin{lemma}\label{lemma2.9}
Assume that $\Omega$ is open, bounded, and has a Lipschitz boundary.  In addition, assume that $\beta\in(n-p,n)$.  Then it holds that
\begin{equation}
W^{1,(p,\beta)}\left(\Omega;\mathbb{R}^{N}\right)\subseteq\mathscr{C}^{0,1-\frac{n-\beta}{p}}\left(\overline{\Omega};\mathbb{R}^{N}\right).\notag
\end{equation}
\end{lemma}

\begin{lemma}\label{lemma2.10}
Let $\Omega$ be a bounded open set without internal cusps, and let $n<\lambda\le n+p$.  Then
\begin{equation}
\mathscr{L}^{p,\lambda}(\Omega)\cong\mathscr{C}^{0,\alpha}\big(\overline{\Omega}\big),\notag
\end{equation}
where $\displaystyle\alpha:=\frac{\lambda-n}{p}$.
\end{lemma}

\begin{lemma}\label{lemma2.20}
Provided that $t\ge r$, then the inclusion
\begin{equation}
L^{t,\alpha}(\Omega)\subseteq L^{r,\alpha}(\Omega)\notag
\end{equation}
is true for any $0\le\alpha<n$.
\end{lemma}

\section{Main Results}

We now prove our regularity results.  We begin by treating the asymptotically convex case.  Then, essentially as a corollary since only a few changes are required, we prove regularity in the convex case.  Finally, as a corollary to this latter result we, at last, establish the regularity of weak solutions to the $p$-Laplacian PDE with VMO coefficients.

\begin{theorem}\label{theorem3.1}
Let $u\in W_{\text{loc}}^{1,p}(\Omega)$ be a local minimizer of the functional \eqref{eq1.3a} and assume that each of condition (A1)--(A3) holds.  Then, letting $\beta$ be the number in condition (A3), for each $n-p<\alpha<\beta$ it holds that
\begin{equation}
u\in\mathscr{C}_{\text{loc}}^{0,\frac{n-\alpha}{p}}(\Omega).\notag
\end{equation}
\end{theorem}

\begin{proof}
The proof of this result builds on the techniques used in the proof of \cite[Theorem 3.1]{goodrich5}.  Therefore, the parts of this proof that are not materially different from \cite{goodrich5} will only be sketched here.  In particular, the main dissimilarities versus \cite[Theorem 3.1]{goodrich5} are the incorporation of the asymptotical relatedness condition and the successful introduction of the VMO condition since here, unlike in \cite{goodrich5}, we do not have access to any particular differentiability of $a$, and so, cannot use Poincar\'{e}'s inequality in the crucial way it was used in \cite[Theorem 3.1]{goodrich5}.

Let us also recall that $u$ will be the minimizer of
\begin{equation}
\int_{\Omega}f(x,u,Du)\ dx,\notag
\end{equation}
whereas $v$ a minimizer of the \textquotedblleft frozen\textquotedblright\ functional
\begin{equation}
\int_{\mathcal{B}_R}g\big((a)_R|Dv|\big)\ dx.\notag
\end{equation}
In addition, we will assume that $v\equiv u$ on $\partial\mathcal{B}_R$.  We first write, essentially exactly as in \cite[(3.2)]{goodrich5},
\begin{equation}\label{eq3.2}
\begin{split}
\int_{\mathcal{B}_{\rho}}|Du|^p\ dx&\le C\int_{\mathcal{B}_{\rho}}|Dv|^p\ dx+C\int_{\mathcal{B}_{\rho}}|Du-Dv|^p\ dx\\
&\le C\rho^n\Vert Dv\Vert_{L^{\infty}\left(\mathcal{B}_{\rho}\right)}+C\int_{\mathcal{B}_R}|Du-Dv|^2\left(|Du|^{p-2}+|Dv|^{p-2}\right)\ dx\\
&\le C\rho^n\dashint_{\mathcal{B}_R}g\big((a)_R|Dv|\big)\ dx+C\int_{\mathcal{B}_R}|Du-Dv|^2\left(|Du|^{p-2}+|Dv|^{p-2}\right)\ dx\\
&\le C\left(\frac{\rho}{R}\right)^n\int_{\mathcal{B}_R}g\big((a)_R|Dv|\big)\ dx+C\int_{\mathcal{B}_R}\Big\{g\big((a)_R|Du|\big)-g\big((a)_R|Dv|\big)\Big\}\ dx,
\end{split}
\end{equation}
where we have used \cite[Lemma 8.6]{giusti1}, recalling that $g\in\mathscr{C}^2(\mathbb{R})$; in addition, we have used the fact that
\begin{equation}
\int_{\mathcal{B}_R}\frac{\partial}{\partial\xi}g\big((a)_R|Dv|\big):[Du-Dv]\ dx=0\notag
\end{equation}
since $v$ satisfies a suitable Euler-Lagrange equation.  Now, adding and subtracting $g\big(a(x)|Du|\big)$ in the second integral on the right-hand side above, we find from \eqref{eq3.2} that
\begin{equation}\label{eq3.3}
\begin{split}
\int_{\mathcal{B}_{\rho}}|Du|^p\ dx&\le C\left(\frac{\rho}{R}\right)^n\int_{\mathcal{B}_R}g\big((a)_R|Dv|\big)\ dx+C\int_{\mathcal{B}_R}\left\{g\big((a)_R|Du|\big)-g\big(a(x)|Du|\big)\right\}\ dx\\
&\quad\quad\quad\quad\quad\quad\quad\quad\quad\quad\quad\quad\quad\quad\quad\quad+C\int_{\mathcal{B}_R}\Big\{g\big(a(x)|Du|\big)-g\big((a)_R|Dv|\big)\Big\}\ dx.
\end{split}
\end{equation}

Next we recall the asymptotic relatedness condition that $f$ and $g$ jointly satisfy.  In particular, for each $\varepsilon>0$ it holds that
\begin{equation}
\Big|f(x,u,Du)-g\big(a(x)|Du|\big)\Big|<\varepsilon|Du|^{p},\notag
\end{equation}
whenever $|Du|>\sigma_{\varepsilon}(x)$.  With this in mind we then estimate from above the second and third terms appearing on the right-hand side \eqref{eq3.3} by writing
\begin{equation}\label{eq3.3a}
\begin{split}
&C\int_{\mathcal{B}_R}\left\{g\big((a)_R|Du|\big)-g\big(a(x)|Du|\big)\right\}\ dx+C\int_{\mathcal{B}_R}\Big\{g\big(a(x)|Du|\big)-g\big((a)_R|Dv|\big)\Big\}\ dx\\
&\le C\int_{\mathcal{B}_R}\left\{g\big((a)_R|Du|\big)-g\big(a(x)|Du|\big)\right\}\ dx\\
&\quad\quad+C\int_{\mathcal{B}_R}\Big\{g\big(a(x)|Du|\big)-f(x,u,Du)+f(x,u,Du)-g\big((a)_R|Dv|\big)\Big\}\ dx\\
&\le C\int_{\mathcal{B}_R}\int_0^1\left|\frac{\partial}{\partial F}f\left(\Big[a(x)+s\big((a)_R-a(x)\big)\Big]|Du|\right)\right|\big|a(x)-(a)_R\big|\ ds\ dx\\
&\quad\quad+C\int_{\mathcal{B}_R\cap\{x\ : \ |Du|\le\sigma_{\varepsilon}(x)\}}\Big|f(x,u,Du)-g\big(a(x)|Du|\big)\Big|\ dx\\
&\quad\quad\quad\quad+C\int_{\mathcal{B}_R\cap\{x\ : \ |Du|>\sigma_{\varepsilon}(x)\}}\Big|f(x,u,Du)-g\big(a(x)|Du|\big)\Big|\ dx\\
&\quad\quad\quad\quad\quad\quad+C\int_{\mathcal{B}_R}f(x,v,Dv)-g\big(a(x)|Dv|\big)+g\big(a(x)|Dv|\big)-g\big((a)_R|Dv|\big)\ dx\\
&\le C\int_{\mathcal{B}_R}\big(1+|Du|^{p-1}\big)|Du|\big|a(x)-(a)_R\big|\ dx\\
&\quad\quad+C\int_{\mathcal{B}\cap\{x\ : \ |Du|\le\sigma_{\varepsilon}(x)\}}\Big(1+\mu(x)+\big(\sigma_{\varepsilon}(x)\big)^p\Big)\ dx\\
&\quad\quad\quad\quad+C\varepsilon\int_{\mathcal{B}_R}|Du|^p\ dx\\
&\quad\quad\quad\quad\quad\quad+C\int_{\mathcal{B}_R}\Big|f(x,v,Dv)-g\big(a(x)|Dv|\big)\Big|\ dx\\
&\quad\quad\quad\quad\quad\quad\quad\quad+C\int_{\mathcal{B}_R}g\big(a(x)|Dv|\big)-g\big((a)_R|Dv|\big)\ dx\\
&\le C\int_{\mathcal{B}_R}\big(1+|Du|^{p-1}\big)|Du|\big|a(x)-(a)_R\big|\ dx+CR^n+CR^{\beta}\\
&\quad\quad+C\varepsilon\int_{\mathcal{B}_R}|Du|^p\ dx+C\varepsilon\int_{\mathcal{B}_R}|Dv|^p\ dx\\
&\quad\quad\quad\quad+C\int_{\mathcal{B}_R}g\big(a(x)|Dv|\big)-g\big((a)_R|Dv|\big)\ dx,
\end{split}
\end{equation}
where in \eqref{eq3.3a} we use the minimality of $u$ to execute the inequality
\begin{equation}
\int_{\mathcal{B}_R}f(x,u,Du)\ dx\le\int_{\mathcal{B}_R}f(x,v,Dv)\ dx.\notag
\end{equation}
We have also used that for $x\in\big\{x\ : \ |Du|\le\sigma_{\varepsilon}(x)\big\}$ it holds that
\begin{equation}
\begin{split}
f(x,u,Du)-g\big(a(x)|Du|\big)&\le\big|f(x,u,Du)\big|+\Big|g\big(a(x)|Du|\big)\Big|\\
&\le \mu(x)+C+C|Du|^p\\
&\le \mu(x)+C+C\big(\sigma_{\varepsilon}(x)\big)^p,\notag
\end{split}
\end{equation}
by means of the growth conditions on the functions $f$ and $g$.  Also, using that $\sigma_{\varepsilon}\in L_{\text{loc}}^{p,\beta}(\Omega)$ we have estimated
\begin{equation}
\int_{\mathcal{B}_R}\big(\sigma_{\varepsilon}(x)\big)^p\ dx\le CR^{\beta},\notag
\end{equation}
and, likewise, using that $\mu\in L^{1,\beta}(\Omega)$ we have estimated
\begin{equation}
\int_{\mathcal{B}_R}\big|\mu(x)\big|\ dx\le CR^{\beta}.\notag
\end{equation}

To further refine the inequality in the preceding paragraph, we begin by estimating
\begin{equation}\label{eq3.4}
\begin{split}
C\varepsilon\int_{\mathcal{B}_R}|Dv|^p\ dx\le C\varepsilon\int_{\mathcal{B}_R}a_1|Dv|^p\ dx&\le C\varepsilon\int_{\mathcal{B}_R}(a)_R|Dv|^p\ dx\\
&\le C\varepsilon\int_{\mathcal{B}_R}g\big((a)_R|Dv|\big)\ dx\\
&\le C\varepsilon\int_{\mathcal{B}_R}g\big((a)_R|Du|\big)\ dx\\
&\le C\varepsilon\int_{\mathcal{B}_R}\left(1+(a)_R^p|Du|^p\right)\ dx\\
&\le CR^n+C\varepsilon\int_{\mathcal{B}_R}|Du|^p\ dx,
\end{split}
\end{equation}
using minimality of $v$.  So, putting \eqref{eq3.4} into \eqref{eq3.3a} we obtain that
\begin{equation}\label{eq3.5}
\begin{split}
\int_{\mathcal{B}_{\rho}}|Du|^p\ dx&\le C\left(\frac{\rho}{R}\right)^n\int_{\mathcal{B}_R}g\big((a)_R|Dv|\big)\ dx+CR^n+CR^{\beta}\\
&\quad\quad+C\int_{\mathcal{B}_R}\big(1+|Du|^{p-1}\big)|Du|\big|a(x)-(a)_R\big|\ dx\\
&\quad\quad\quad\quad+C\varepsilon\int_{\mathcal{B}_R}|Du|^p\ dx\\
&\quad\quad\quad\quad\quad\quad+C\int_{\mathcal{B}_R}g\big(a(x)|Dv|\big)-g\big((a)_R|Dv|\big)\ dx\\
&\le C\left(\frac{\rho}{R}\right)^n\int_{\mathcal{B}_R}|Du|^p\ dx+CR^n+CR^{\beta}\\
&\quad\quad+C\int_{\mathcal{B}_R}\big(1+|Du|^{p-1}\big)|Du|\big|a(x)-(a)_R\big|\ dx\\
&\quad\quad\quad\quad+C\varepsilon\int_{\mathcal{B}_R}|Du|^p\ dx\\
&\quad\quad\quad\quad\quad\quad+C\int_{\mathcal{B}_R}g\big(a(x)|Dv|\big)-g\big((a)_R|Dv|\big)\ dx,
\end{split}
\end{equation}
where we have used the calculation
\begin{equation}
\begin{split}
C\left(\frac{\rho}{R}\right)^n\int_{\mathcal{B}_R}g\big((a)_R|Dv|\big)\ dx&\le C\left(\frac{\rho}{R}\right)^n\int_{\mathcal{B}_R}g\big((a)_R|Du|\big)\ dx\\
&\le C\left(\frac{\rho}{R}\right)^n\int_{\mathcal{B}_R}\left(1+a_1^p|Du|^p\right)\ dx\\
&\le CR^n+C\left(\frac{\rho}{R}\right)^n\int_{\mathcal{B}_R}|Du|^p\ dx,\notag
\end{split}
\end{equation}
using the minimality of $v$ to obtain the first inequality.  On the other hand, regarding the quantity
\begin{equation}
C\int_{\mathcal{B}_R}g\big(a(x)|Dv|\big)-g\big((a)_R|Dv|\big)\ dx,\notag
\end{equation}
which appears on the right-hand side of \eqref{eq3.5}, we estimate this by writing, just as in \cite{goodrich5},
\begin{equation}
C\int_{\mathcal{B}_R}g\big(a(x)|Dv|\big)-g\big((a)_R|Dv|\big)\ dx\le C\int_{\mathcal{B}_R}\big(1+|Dv|^{p-1}\big)|Dv|\big|a(x)-(a)_R\big|\ dx.\notag
\end{equation}

So, all in all, we obtain the inequality
\begin{equation}\label{eq3.6}
\begin{split}
\int_{\mathcal{B}_{\rho}}|Du|^p\ dx&\le C\left(\frac{\rho}{R}\right)^n\int_{\mathcal{B}_R}|Du|^p\ dx+CR^n+CR^{\beta}\\
&\quad\quad+C\int_{\mathcal{B}_R}\big(1+|Du|^{p-1}\big)|Du|\big|a(x)-(a)_R\big|\ dx\\
&\quad\quad\quad\quad+C\varepsilon\int_{\mathcal{B}_R}|Du|^p\ dx\\
&\quad\quad\quad\quad\quad\quad+C\int_{\mathcal{B}_R}g\big(a(x)|Dv|\big)-g\big((a)_R|Dv|\big)\ dx\\
&\le C\int_{\mathcal{B}_R}\big(1+|Du|^{p-1}\big)|Du|\big|a(x)-(a)_R\big|\ dx\\
&\quad\quad+C\int_{\mathcal{B}_R}\big(1+|Dv|^{p-1}\big)|Dv|\big|a(x)-(a)_R\big|\ dx\\
&\quad\quad\quad\quad+CR^{\beta}+C\left[\varepsilon+\left(\frac{\rho}{R}\right)^n\right]\int_{\mathcal{B}_R}|Du|^p\ dx,
\end{split}
\end{equation}
using that
\begin{equation}
R^n<R^{\beta}.\notag
\end{equation}
Now, choose some number $q>1$ and note that
\begin{equation}\label{eq3.6a}
\begin{split}
C\int_{\mathcal{B}_R}\big|a(x)-(a)_R\big|\ dx\le CR^n\dashint_{\mathcal{B}_R}\big|a(x)-(a)_R\big|\ dx&\le CR^n\left(\dashint_{\mathcal{B}_R}\big|a(x)-(a)_R\big|^{q}\ dx\right)^{\frac{1}{q}}\\
&\le CR^n\Vert a\Vert_*,
\end{split}
\end{equation}
where we have used Lemma \ref{lemma2.6}.  Let $\varepsilon$ be some number such that $\varepsilon\le\varepsilon_0$, where $\varepsilon_0$ is the number from the higher integrability lemma for $u$.  Then using inequality \eqref{eq3.6a} together with the estimates in \cite{goodrich5} we note that
\begin{equation}\label{eq3.7}
\begin{split}
&C\int_{\mathcal{B}_R}\big(1+|Du|^{p-1}\big)|Du|\big|a(x)-(a)_R\big|\ dx\\
&\quad\quad\le C\int_{\mathcal{B}_R}|Du|^p\big|a(x)-(a)_R\big|\ dx+C\int_{\mathcal{B}_R}\big|a(x)-(a)_R\big|\ dx\\
&\quad\quad\le CR^n\underbrace{\left(\dashint_{\mathcal{B}_R}\big|a(x)-(a)_R\big|^{\frac{p+\varepsilon}{\varepsilon}}\ dx\right)^{\frac{\varepsilon}{p+\varepsilon}}}_{\le C\Vert a\Vert_*}\left(\dashint_{\mathcal{B}_R}|Du|^{p+\varepsilon}\ dx\right)^{\frac{p}{p+\varepsilon}}+CR^n\Vert a\Vert_*\\
&\quad\quad\le CR^n\Vert a\Vert_*\left(\dashint_{\mathcal{B}_R}|Du|^{p+\varepsilon}\ dx\right)^{\frac{p}{p+\varepsilon}}+CR^n\Vert a\Vert_*\\
&\quad\quad\le CR^n\Vert a\Vert_*\left[\left(\dashint_{\mathcal{B}_R}|Du|^{p+\varepsilon}\ dx\right)^{\frac{p}{p+\varepsilon}}+1\right],
\end{split}
\end{equation}
using again Lemma \ref{lemma2.6} since $\displaystyle\frac{p+\varepsilon}{\varepsilon}>1$.  And then in a similar way, once more using Lemma \ref{lemma2.6}, we also obtain that
\begin{equation}\label{eq3.8}
C\int_{\mathcal{B}_R}\big(1+|Dv|^{p-1}\big)|Dv|\big|a(x)-(a)_R\big|\ dx\le CR^n\Vert a\Vert_*\left[\left(\dashint_{\mathcal{B}_R}|Dv|^{p+\varepsilon}\ dx\right)^{\frac{p}{p+\varepsilon}}+1\right].
\end{equation}

In addition, the boundary reverse H\"{o}lder inequality for $v$, namely Lemma \ref{lemma2.17}, implies that
\begin{equation}\label{eq3.9}
\begin{split}
\dashint_{\mathcal{B}_R}|Dv|^{p+\varepsilon}\ dx&\le C\left(\dashint_{\mathcal{B}_R}|Du-Dv|^p\ dx\right)^{\frac{p+\varepsilon}{p}}+C\dashint_{\mathcal{B}_{2R}}\left(1+|Du|^{p+\varepsilon}\right)\ dx\\
&\le C\left(\dashint_{\mathcal{B}_R}|Du|^p\ dx+\dashint_{\mathcal{B}_R}|Dv|^p\ dx\right)^{\frac{p+\varepsilon}{p}}+C\dashint_{\mathcal{B}_{2R}}\left(1+|Du|^{p+\varepsilon}\right)\ dx\\
&\le C+C\left(1+\dashint_{\mathcal{B}_R}|Du|^p\ dx\right)^{\frac{p+\varepsilon}{p}}+C\dashint_{\mathcal{B}_{2R}}|Du|^{p+\varepsilon}\ dx\\
&\le C+C\left(\dashint_{\mathcal{B}_R}|Du|^p\ dx\right)^{\frac{p+\varepsilon}{p}}\\
&\quad\quad\quad\quad\quad\quad\quad+\underbrace{C\left[\left(\dashint_{\mathcal{B}_{4R}}|Du|^p\ dx\right)^{\frac{p+\varepsilon}{p}}+\dashint_{\mathcal{B}_{4R}}\big(1+A(x)\big)^{\frac{p+\varepsilon}{p}}\ dx\right]}_{\text{by Lemma }\ref{lemma2.16aaa}}\\
&\le C+CR^{\beta-n}+C\left(\dashint_{\mathcal{B}_{4R}}|Du|^{p}\ dx\right)^{\frac{p+\varepsilon}{p}},
\end{split}
\end{equation}
where, as earlier, we have used the minimality of $v$ to write
\begin{equation}
\begin{split}
\int_{\mathcal{B}_R}|Dv|^p\ dx\le C\int_{\mathcal{B}_R}1+a_1|Dv|^p\ dx&\le C\int_{\mathcal{B}_R}g\big((a)_R|Dv|\big)\ dx\\
&\le C\int_{\mathcal{B}_R}g\big((a)_R|Du|\big)\ dx\\
&\le CR^n+C\int_{\mathcal{B}_R}|Du|^p\ dx\notag
\end{split}
\end{equation}
in \eqref{eq3.9} above.  Note, furthermore, that in \eqref{eq3.9} we have used that
\begin{equation}
\dashint_{\mathcal{B}_R}\big(1+A(x)\big)^{\frac{p+\varepsilon}{p}}\ dx\le C+C\dashint_{\mathcal{B}_R}\big(A(x)\big)^{\frac{p+\varepsilon}{p}}\ dx\le C+CR^{\beta-n},\notag
\end{equation}
using that $A\in L^{1+\eta_0,\beta}(\Omega)$ for some $\eta_0>0$ so that by assuming that $q$ is chosen sufficiently close to $p$ such that $1<\displaystyle\frac{q}{p}<1+\eta_0$ it follows from Lemma \ref{lemma2.20} that
\begin{equation}
L^{1+\eta_0,\beta}(\Omega)\subseteq L^{\frac{p+\varepsilon}{p},\beta}(\Omega).\notag
\end{equation}

Therefore, from \eqref{eq3.9} we calculate
\begin{equation}\label{eq3.10}
\begin{split}
\left(\dashint_{\mathcal{B}_R}|Dv|^{p+\varepsilon}\ dx\right)^{\frac{p}{p+\varepsilon}}&\le C\left[C+CR^{\beta-n}+C\left(\dashint_{\mathcal{B}_{4R}}|Du|^p\ dx\right)^{\frac{p+\varepsilon}{p}}\right]^{\frac{p}{p+\varepsilon}}\\
&\le C\big[1+R^{\beta-n}\big]^{\frac{p}{p+\varepsilon}}+C\left(\dashint_{\mathcal{B}_{4R}}|Du|^p\ dx\right)\\
&\le C+CR^{\beta-n}+CR^{-n}\int_{\mathcal{B}_{4R}}|Du|^p\ dx.
\end{split}
\end{equation}
On the other hand, the reverse H\"{o}lder inequality for $u$ implies that
\begin{equation}\label{eq3.11}
\begin{split}
\left(\dashint_{\mathcal{B}_R}|Du|^{p+\varepsilon}\right)^{\frac{p}{p+\varepsilon}}\ dx&\le C\left[\left(\dashint_{\mathcal{B}_{2R}}|Du|^p\ dx\right)^{\frac{p+\varepsilon}{p}}+\dashint_{\mathcal{B}_{2R}}\big(1+A(x)\big)^{\frac{p+\varepsilon}{p}}\ dx\right]^{\frac{p}{p+\varepsilon}}\\
&\le C\left[\left(\dashint_{\mathcal{B}_{2R}}|Du|^p\ dx\right)^{\frac{p+\varepsilon}{p}}+C+CR^{\beta-n}\right]^{\frac{p}{p+\varepsilon}}\\
&\le C+CR^{\beta-n}+C\left(\dashint_{\mathcal{B}_{2R}}|Du|^p\ dx\right)\\
&\le C+CR^{\beta-n}+CR^{-n}\left(\int_{\mathcal{B}_{2R}}|Du|^p\ dx\right).
\end{split}
\end{equation}

Thus, using estimates \eqref{eq3.7}--\eqref{eq3.8} and \eqref{eq3.10}--\eqref{eq3.11} in inequality \eqref{eq3.6} we conclude that
\begin{equation}\label{eq3.12}
\begin{split}
&\int_{\mathcal{B}_{\rho}}|Du|^p\ dx\\
&\le CR^{\beta}+C\left[\varepsilon+\left(\frac{\rho}{R}\right)^n\right]\int_{\mathcal{B}_R}|Du|^p\ dx\\
&\quad\quad+CR^n\Vert a\Vert_*\left[\left(\dashint_{\mathcal{B}_R}|Du|^{p+\varepsilon}\ dx\right)^{\frac{p}{p+\varepsilon}}+1\right]+CR^n\Vert a\Vert_*\left[\left(\dashint_{\mathcal{B}_R}|Dv|^{p+\varepsilon}\ dx\right)^{\frac{p}{p+\varepsilon}}+1\right]\\
&\le CR^{\beta}+C\left[\varepsilon+\left(\frac{\rho}{R}\right)^n\right]\int_{\mathcal{B}_R}|Du|^p\ dx\\
&\quad\quad+CR^n\Vert a\Vert_*\underbrace{\left[C+CR^{\beta-n}+CR^{-n}\left(\int_{\mathcal{B}_{2R}}|Du|^p\ dx\right)+1\right]}_{\text{from }\eqref{eq3.11}}\\
&\quad\quad+CR^n\Vert a\Vert_*\underbrace{\left[C+CR^{\beta-n}+CR^{-n}\int_{\mathcal{B}_{4R}}|Du|^p\ dx+1\right]}_{\text{from }\eqref{eq3.10}}\\
&\le CR^{\beta}+C\left[\varepsilon+\left(\frac{\rho}{R}\right)^n+\Vert a\Vert_*\right]\int_{\mathcal{B}_{4R}}|Du|^p\ dx+C\Vert a\Vert_*\left[R^{\beta}+R^{n}\right]\\
&\le CR^{\beta}\big(1+\Vert a\Vert_*\big)+C\left[\varepsilon+\left(\frac{\rho}{R}\right)^n+\Vert a\Vert_*\right]\int_{\mathcal{B}_{4R}}|Du|^p\ dx\\
&\le CR^{\beta}+C\left[\varepsilon+\left(\frac{\rho}{R}\right)^n+\Vert a\Vert_*\right]\int_{\mathcal{B}_{4R}}|Du|^p\ dx,
\end{split}
\end{equation}
where to obtain the final inequality in \eqref{eq3.12} we have used that
\begin{equation}
\Vert a\Vert_*\le1,\notag
\end{equation}
for all $R>0$ sufficiently small, seeing as $a\in VMO(\Omega)$.  So, in summary, from \eqref{eq3.12} we obtain
\begin{equation}\label{eq3.14}
\int_{\mathcal{B}_{\rho}}|Du|^p\ dx\le CR^{\beta}+C\left[\varepsilon+\left(\frac{\rho}{R}\right)^n+\Vert a\Vert_*\right]\int_{\mathcal{B}_{4R}}|Du|^p\ dx.
\end{equation}

We now complete the proof by means of a standard inductive iteration argument, proceeding from the initial estimate \eqref{eq3.14} above.  Since this argument is materially really no different than the versions that can be found, for example, in \cite{fossgoodrich1,goodrich1}, we only briefly sketch the idea for the sake of completeness.  In particular, define $\varphi\ : \ [0,+\infty)\rightarrow[0,+\infty)$ by
\begin{equation}
\varphi(R):=\int_{\mathcal{B}_R}|Du|^p\ dx\notag
\end{equation}
so that \eqref{eq3.14} can be rewritten in the form
\begin{equation}\label{eq3.16ggg}
\varphi(\rho)\le C(4R)^{\beta}+C\left[\varepsilon+\left(\frac{\rho}{4R}\right)^n+\Vert a\Vert_*\right]\varphi(4R).
\end{equation}
Now recalling that $x\mapsto a(x)$ is of class VMO it follows that for any $\varepsilon_0>0$ given, there exists $R_0>0$ sufficiently small such that whenever $0<R\le R_0$ it follows that
\begin{equation}
\dashint_{\mathcal{B}_R}\big|a(x)-(a)_R\big|\ dx<\varepsilon_0.\notag
\end{equation}
In particular, this means that \eqref{eq3.16ggg} can be recast in the form
\begin{equation}\label{eq3.17ggg}
\varphi(\rho)\le C(4R)^{\beta}+C\left[2\varepsilon+\left(\frac{\rho}{4R}\right)^n\right]\varphi(4R),
\end{equation}
for all $R>0$ sufficiently small.  But then \eqref{eq3.17ggg} can be inductively iterated by appealing to \cite[Lemma 2.1, p. 86]{giaquinta1}.  This will yield that
\begin{equation}
u\in\mathscr{C}_{\text{loc}}^{0,1-\frac{n-\alpha}{p}}(\Omega)\notag
\end{equation}
for each $\alpha\in(n-p,\beta)$.  And this completes the proof.
\end{proof}

We now give a simple example for the sake of illustrating the basic application of Theorem \ref{theorem3.1}.

\begin{example}\label{example3.2}
To illustrate the application of the preceding result we consider the functional
\begin{equation}
\int_{\Omega}\underbrace{a(x)|Du|^p+b(x)|Du|^r}_{=:f(x,u,Du)}\ dx,\notag
\end{equation}
where $r<p$.  Note that
\begin{equation}
\big|f(x,u,Du)-a(x)|Du|^p\big|=b(x)|Du|^r<\varepsilon|Du|^p\notag
\end{equation}
whenever
\begin{equation}\label{eq3.15}
|Du|>\left(\frac{1}{\varepsilon}b(x)\right)^{\frac{1}{p-r}}=:\sigma_{\varepsilon}(x).
\end{equation}

Recall that $\sigma_{\varepsilon}$ was assumed to satisfy $\sigma_{\varepsilon}\in L^{\gamma,\beta}(\Omega)$ for some $\gamma>p$ and $\beta\in(n-p,n)$.  Therefore, from our definition of $\sigma_{\varepsilon}$ in \eqref{eq3.15} we need to require that
\begin{equation}\label{eq3.16}
\sup_{\substack{0<\rho<\text{dist}(x,\partial\Omega)\\ x\in\Omega}}\frac{1}{\rho^{\beta}}\int_{\mathcal{B}_{\rho}(x)}\big|b(y)\big|^{\frac{\gamma}{p-r}}\ dy<+\infty,
\end{equation}
where we need
\begin{equation}
\frac{\gamma}{p-r}>p,\notag
\end{equation}
or, equivalently, that
\begin{equation}
\gamma>p(p-r).\notag
\end{equation}
Note that the above condition \eqref{eq3.16} on the map $x\mapsto b(x)$ does not require it to be continuous or, for that matter, even bounded.

Finally, that $f$ satisfies the desired growth condition follows from the calculation
\begin{equation}
\begin{split}
\big|f(x,u,\xi)\big|\le a(x)|\xi|^p+b(x)|\xi|^r&\le a(x)|\xi|^p+\frac{p-r}{p}\big(b(x)\big)^{\frac{p}{p-r}}+\frac{r}{p}|\xi|^p\\
&\le\left(a(x)+\frac{r}{p}\right)|\xi|^p+\frac{p-r}{p}\big(b(x)\big)^{\frac{p}{p-r}},\notag
\end{split}
\end{equation}
where we have used Young's inequality.  Since $x\mapsto a(x)$ is bounded, the above inequality may be recast in the form
\begin{equation}\label{eq3.17}
\big|f(x,u,\xi)\big|\le\frac{p-r}{p}\big(b(x)\big)^{\frac{p}{p-r}}+C\big(1+|\xi|^p\big):=\mu(x)+C\big(1+|\xi|^p\big).
\end{equation}
By virtue of the earlier calculation and the fact that $\displaystyle\frac{\gamma}{p-r}>\frac{p}{p-r}$, we have that
\begin{equation}
\mu\in L^{\sigma,\beta}(\Omega),\notag
\end{equation}
for some $\sigma>1$.  So, it follows from inequality \eqref{eq3.17} that $f$ satisfies the required growth condition.
\end{example}

We now provide a corollary to Theorem \ref{theorem3.1}.  This result, Corollary \ref{corollary3.3}, treats the case in which we consider minimizers of the functional
\begin{equation}
\int_{\Omega}g\big(a(x)|Du|\big)\ dx,\notag
\end{equation}
where $g$ satisfies hypothesis (A1) and $a$ satisfies hypothesis (A2).

\begin{corollary}\label{corollary3.3}
Let $u\in W_{\text{loc}}^{1,p}(\Omega)$ be a local minimizer of the functional \eqref{eq1.3} and assume that conditions (A1)--(A2) are true.
Then for each $0<\alpha<1$ it holds that
\begin{equation}
u\in\mathscr{C}_{\text{loc}}^{0,\alpha}\big(\Omega\big).\notag
\end{equation}
\end{corollary}

\begin{proof}
The proof of this result does not differ too significantly from that of Theorem \ref{theorem3.1}.  Consequently, we merely point out the relevant changes rather than showing every detail.  So, we note that essentially repeating the steps of the proof of Theorem \ref{theorem3.1} we arrive at the analogy of estimate \eqref{eq3.5}, which now reads
\begin{equation}\label{eq3.19a}
\begin{split}
\int_{\mathcal{B}_{\rho}}|Du|^p\ dx&\le CR^n+C\left(\frac{\rho}{R}\right)^n\int_{\mathcal{B}_R}|Du|^p\ dx\\
&\quad\quad+C\int_{\mathcal{B}_R}\big(1+|Du|^{p-1}\big)|Du|\big|a(x)-(a)_R\big|\ dx\\
&\quad\quad\quad\quad+C\int_{\mathcal{B}_R}\big(1+|Dv|^{p-1}\big)|Dv|\big|a(x)-(a)_R\big|\ dx,
\end{split}
\end{equation}
owing to the fact that here we do not use an asymptotical relatedness condition, and so, we acquire neither the term $R^{\beta}$ nor the term
\begin{equation}
C\varepsilon\int_{\mathcal{B}_R}|Du|^p\ dx,\notag
\end{equation}
and, moreover, the term
\begin{equation}
C\int_{\mathcal{B}_R}g\big(a(x)|Dv|\big)-g\big((a)_R|Dv|\big)\ dx\notag
\end{equation}
is replaced by
\begin{equation}
C\int_{\mathcal{B}_R}\big(1+|Dv|^{p-1}\big)|Dv|\big|a(x)-(a)_R\big|\ dx.\notag
\end{equation}
But now \eqref{eq3.19a} may be estimated from above by repeating the same basic steps as in the proof of Theorem \ref{theorem3.1} -- that is, essentially the same application of the reverse H\"{o}lder inequalities for both $u$ and $v$.  All in all, this yields the analogue of estimate \eqref{eq3.14}, which is
\begin{equation}\label{eq3.20a}
\int_{\mathcal{B}_{\rho}}|Du|^p\ dx\le C(4R)^n\Vert a\Vert_*+C\bigg[\left(\frac{\rho}{4R}\right)^{n}+\Vert a\Vert_*\bigg]\int_{\mathcal{B}_{4R}}|Du|^p\ dx.
\end{equation}

So, in comparing \eqref{eq3.20a} to \eqref{eq3.14} we see that the only material difference is that in \eqref{eq3.20a} above we neither have the term $R^{\beta}$ nor the quantity $\varepsilon$ appearing in the factor multiplying $\displaystyle\int_{\mathcal{B}_{4R}}|Du|^p\ dx$.  Other than these two alterations, the two estimates are identical.  The only change that this induces is that whereas in Theorem \ref{theorem3.1} the amount of Morrey regularity that $\sigma_{\varepsilon}$ enjoys is encoded by the exponent $\beta$ and, hence, subsequently affects the degree of regularity of $u$, in this case there is no such quantity, and so, we conclude, by the same sort of inductive iteration argument as in Theorem \ref{theorem3.1}, namely by again using Giaquinta \cite[Lemma 2.1, p. 86]{giaquinta1}, that
\begin{equation}
u\in\mathscr{C}_{\text{loc}}^{0,\alpha}(\Omega)\notag
\end{equation}
for every $0<\alpha<1$.  And this completes the proof.
\end{proof}

We conclude by applying the previous results to the elliptic systems \eqref{eq1.1} and \eqref{eq1.2}, which shows that weak solutions $u\in W_{\text{loc}}^{1,p}(\Omega)$ to the PDE
\begin{equation}
\nabla\cdot\left(a(x)f'\big(a(x)|Du|\big)\frac{Du}{|Du|}\right)=0\text{, }x\in\Omega\notag
\end{equation}
are also locally H\"{o}lder continuous under assumptions (A1)--(A2).

\begin{corollary}\label{corollary3.4}
Let $u\in W_{\text{loc}}^{1,p}(\Omega)$ be a weak solution of PDE \eqref{eq1.2}.  Assume that conditions (A1)--(A2) hold.
Then for each $0<\alpha<1$ is holds that
\begin{equation}
u\in\mathscr{C}_{\text{loc}}^{0,\alpha}\big(\Omega\big).\notag
\end{equation}
\end{corollary}

\begin{proof}
Follows at once from the $p$-convexity assumption imposed on $g$ -- see \cite{evans1} for additional details.
\end{proof}

\begin{remark}\label{remark3.5}
We mention that in the special case in which we select
\begin{equation}
g(t):=t^p,\notag
\end{equation}
it follows that solutions of the $p$-Laplacian system
\begin{equation}
\nabla\cdot\big(a(x)|Du|^{p-2}Du\big)=0\text{, }x\in\Omega,\notag
\end{equation}
where $a\in VMO(\Omega)\cap L^{\infty}(\Omega)$, are of class $\mathscr{C}_{\text{loc}}^{0,\alpha}(\Omega)$ for each $0<\alpha<1$.
\end{remark}

\textbf{Acknowledgements.} The authors are grateful to the referees for the useful suggestions. Part of this work was completed when the first author (CSG) visited the Dipartimento di Matematica e Informatica at the University of Catania.  He would like to express his sincere gratitude for the warm hospitality that he was shown during his visit.  The second author (MAR) is partially supported by I.N.D.A.M - G.N.A.M.P.A. 2019 and the ``RUDN University Program 5-100''.

\end{document}